\documentclass[11pt]{article}
\usepackage{amssymb,amsmath,amsthm}
\usepackage{mathtools}
\usepackage{mathrsfs}
\usepackage{enumerate}
\usepackage{tikz}
\usepackage[margin=1in]{geometry}
\usepackage[backend=biber,style=alphabetic]{biblatex}
\usepackage{hyperref}
\usepackage[nameinlink,noabbrev]{cleveref}
\usetikzlibrary{positioning}
\usetikzlibrary{calc}

\newtheorem{theorem}{Theorem}[section]
\newtheorem{lemma}[theorem]{Lemma}

\newtheorem{corollary}[theorem]{Corollary}

\theoremstyle{definition}
\newtheorem{definition}[theorem]{Definition}
\newtheorem{remark}[theorem]{Remark}

\newtheorem{claim}{Claim}

\newcommand{\sqbinom}[2]{\genfrac{[}{]}{0pt}{}{#1}{#2}}

\title{A $q$-analogue of a binomial supercongruence}

\author{Oscar Gonzalez}

\date{\today}

\begin{document}

\maketitle

\begin{abstract}
In a recent paper \cite{ichino2025representationsmathrmgl2mathbbzpnmathbbzsupercongruences}, Ichino and 
Prasanna proved a supercongruence for binomial coefficients 
and related it to representations of $\mathrm{GL}_2$ over 
$p$-adic rings. We prove a $q$-analog of this supercongruence using the $q$-Lucas theorem, from which a strengthening of the main result of loc. cit. can easily be deduced. Our proof also substantially simplifies the original argument of \cite{ichino2025representationsmathrmgl2mathbbzpnmathbbzsupercongruences}.

\end{abstract}

\section{Introduction}
We will generalize the following theorem. 
\begin{theorem}\cite{ichino2025representationsmathrmgl2mathbbzpnmathbbzsupercongruences} Fix a prime $p>2$ and two non-negative integers $j$ and $n$. For any non-negative integer $r$ with binary expansion 
\begin{equation} r = \sum_{j=0}^\infty r_j 2^j, \quad r_j\in \{ 0, 1 \} \nonumber 
\end{equation}
let $k_r$ be defined by
    \begin{equation}k_r := (p-1)\sum_{j=0}^\infty r_jp^j.\nonumber \end{equation}
Then for any non-negative integer $i$, we have:
    \begin{equation}\label{ak} \sum_{r=0}^{2^n-1}(-1)^r{{k_r/2}\choose{i}}{{-1/2-i}\choose{k_r/2}}\equiv 0\mod p^n.\end{equation}
\end{theorem}
 \par

    We generalize the congruence in \Cref{ak} in two ways: firstly, 
    by proving the $q$-analog and secondly, by generalizing the definition of $k_r$. The main new results are \Cref{main} below (which gives a congruence for cyclotomic polynomials)
    and its corollaries (Cor. \ref{1coro} through \ref{when i=0}).
    
    \par
\subsection{Definition}
We will first generalize the definition of $k_r$. Consider a choice of integers $0\leq a < b\leq p-1$, define $k_r\in \mathbb{Z}_p$ by 
\begin{equation}\label{krdefined} k_r:= \sum_{j=0}^\infty ((b-a)r_j+a)p^j
\end{equation}
where $r_j$ is defined as before. Note that the coefficient $(b-a)r_j+a$ will either be $a$ or $b$ depending on whether $r_j=0$ or $r_j=1$. Observe that if $a\neq 0$ then $k_r\in \mathbb{Z}_p\backslash \mathbb{Z}$. \par

From \Cref{ak} while $-\frac{1}{2}$ may not be an integer, we can instead define it as a $p$-adic integer. Indeed we observe that $-\frac{1}{2}$ is the unique element with all $p$-adic digits $\frac{p-1}{2}$, so for a more general value that depends on the choice of $a$ and $b$, we define in $\mathbb{Z}_p$
\begin{equation}\alpha (a,b):= \sum_{j=0}^\infty \frac{a+b}{2}p^j=\frac{a+b}{2(1-p)}, \quad 2\vert a,\quad 2\vert b, \quad  0\leq a< b\leq p-1 .\nonumber \end{equation}
Indeed, if $a=0$ and $b=p-1$ then $\alpha (0,p-1)=-\frac{1}{2}$. The constraint $2\vert a$ and $2\vert b$ is to ensure that each digit of $k_r/2$ is an integer, each digit being defined as  $\frac{1}{2}((b-a)r_j+a)$ for $p^j$.  \par

We have now seen that we are working with $p$-adic inputs for binomial coefficients. It may be rather difficult to work with such binomial coefficients directly or even define them, so to get around this we will often work with finite truncations of these $p$-adic integers. So we introduce some notation related to this. \par 
\begin{definition}
For a $p$-adic integer $r$ defined by 
    $$r=r_0+r_1p+r_2p^2+\dots = \sum_{i=0}^\infty r_jp^j$$
    we can define the following two functions 
    $$T_m(r): = \sum_{j=m}^\infty r_jp^{j-m}$$
    and 
    $$T_m^\ell(r): = \sum_{j=m}^{m+\ell-1}r_jp^{j-m}.$$
\end{definition}
The first function $T_m(r)$ removes the first $m$ digits and shifts the remaining digits by $m$ spots. The second function $T_m^\ell (r)$ is the first $\ell$ digits of $T_m(r)$, and sets the remaining digits to zero. \par

Now we wish to consider a set of $k_r$ as defined in \Cref{krdefined} for fixed integers $a$ and $b$ so that $2\vert a$ and $2\vert b$ and $0\leq a< b\leq p-1$. Let $K_n=\{ k_r\vert 0\leq r\leq 2^n-1\}$, for our convenience it will be useful to use a second characterization $K_n = \{ \sum_{j=0}^{n-1} s_jp^j+\sum_{j=n}^\infty ap^j\vert s_j\in \{ a,b\} \}$. \par 
Now we wish to define a way of swapping the first $m$ digits of a given $k_r$ by swapping $a$ for $b$ and $b$ for $a$. 
\begin{definition} If $k_r\in K_n$ then $k_r=(b-a)\sum_{j=0}^{n-1} r_jp^j+\sum_{j=0}^\infty ap^j$ we define $r_j^c:=1-r_j$, then for $0\leq m\leq n$ we define 
$$D_m(k_r)=(b-a)\sum_{j=0}^{m-1}r_j^cp^j+(b-a)\sum_{j=m}^{n-1}r_jp^j+\sum_{j=0}^\infty ap^j.$$

Similarly, we define 
$$D_m(k_r/2)=\frac{1}{2}D_m(k_r).$$

\end{definition}
There are three parts here, the first $m$ digits that are swapped, the next $n-m$ digits that stay fixed, and the infinite tail that also stays fixed.\par

We now state some facts that we will make use of for our main proof.
\begin{lemma}\label[lemma]{Basics}
 The following are true for any $1\leq m\leq n.$
\begin{itemize}
\item[(a)] $k_r\in K_n$ iff $D_m(k_r)\in K_n$
\item[(b)] $D_m(D_m(k_r))=k_r$
\item[(c)] $T_0^m(D_m(k_r)/2)+T_0^m(k_r/2)=T_0^m(\alpha (a,b))$
\item[(d)] $T_\ell(k_r)=T_\ell (D_m(k_r))$ for $m\leq \ell\leq n.$ 
\item[(e)] For a given $k_r$ we have $T_0^1(k_r)=a$ iff $2\vert r$ (or equivalently $T_0^1(k_r)=b$ iff $2\not\vert r$). Further, there exists an integer $0\leq s\leq 2^n-1$ such that $k_s=D_m(k_r)$ and $s\not\equiv r \mod 2$. 
\end{itemize}
\end{lemma}

As stated before $k_r$ and $\alpha (a,b)$ are $p$-adic integers and may not be integers. In order to build up to these $p$-adic integers, we will consider finite truncations of them which live in $\mathbb{Z}$. 
\begin{definition} In general, we will choose integers $s$ and $t$ such that $n\leq s$ and $n\leq t$ so that we can define
    $$k_r^{(s)}:= T_0^s(k_r)$$
    and
    $$\alpha ^{(t)}(a,b):= T_0^t(\alpha (a,b))$$
\end{definition}
In a similar manner to how we defined $D_m$ for $k_r$ we define it for $k_r^{(s)}$ as well

\begin{definition} For $1\leq m\leq n\leq s$ 
    $$D_m(k_r^{(s)}):=T_0^s(D_m(k_r)),$$
    and similarly 
    $$D_m(k_r^{(s)}/2)=T_0^s(D_m(k_r/2)).$$
\end{definition}
\begin{lemma}\label[lemma]{BasicTruncate}
All of \Cref{Basics} apply to $k_r^{(s)}$ for any integers $1\leq m\leq n\leq s$.
\end{lemma}
\begin{proof}
This is an immediate consequence of $D_m(T_0^s(k_r))=T_0^s(D_m(k_r))$. 
\end{proof}

\subsection{Binomial Coefficients}
Now the goal will be to prove a $q$-analogue of \Cref{ak}, so we briefly recall standard notation from $q$-series theory from \cite{Gasper_Rahman_2004}\cite{Stanley_1997}. For $k\geq 0$ and arbitrary $x$, the generalized binomial coefficient is defined by 
\begin{equation}
    {{x}\choose{k}}= \frac{x(x-1)\dots (x-k+1)}{k!}, \quad {{x}\choose{0}}=1.\nonumber 
\end{equation}
For $n\geq 0$, the $q$-integer and $q$-factorial are given by 
\begin{equation} [n]_q=  \frac{1-q^n}{1-q}, \quad [n]_q!=\prod_{i=1}^n[i]_q \nonumber 
\end{equation}
with $[0]_q!=1$. We also recall the $q$-Pochhammer symbol

\begin{equation}(x;q)_n:= \prod_{i=0}^{n-1}(1-xq^i), \quad (x;q)_0=1. \nonumber 
\end{equation}

The $q$-binomial (or Gaussian) coefficient is then defined by 
    \begin{equation}
 \sqbinom{n}{k}_q:=\frac{[n]_q!}{[k]_q![n-k]_q!}=\frac{(q;q)_n}{(q;q)_k(q;q)_{n-k}} \quad 0\leq k\leq n.\nonumber 
 \end{equation}
 As $q\rightarrow 1$, one recovers the ordinary binomial coefficient:
 \begin{equation}
 \lim_{q\rightarrow 1} \sqbinom{n}{k}_q = {{n}\choose{k}}.\nonumber 
 \end{equation}

\section{An application of the $q$-binomial analogue of Lucas's Theorem}

\subsection{$q$-Lucas Theorem}
The following result is a $q$-analogue of Lucas's theorem and will play a key role in the proof. Unlike the classical version modulo a prime, it holds for an arbitrary positive integer $k$. For this section and any subsequent section we will refer to the $k$-th cyclotomic polynomial as $\Phi _k(q)$ or $\Phi _k$, further $\zeta _k$ will refer to a primitive root $k$-th root of unity. 
\begin{theorem} \cite[Proposition 2.2]{DESARMENIEN198219} \label{Lucas qbinomial}
    If $a,b,r,s,k$ are integers, where $0\leq b, s< k$, then 

    $$ \sqbinom{ka+b}{kr+s}_q\equiv {{a}\choose{r}}\sqbinom{b}{s}_q \mod \Phi _k.$$
    More specifically $$\sqbinom{ka+b}{kr+s}_q-{{a}\choose{r}}\sqbinom{b}{s}_q\in \Phi _{k}(q)\mathbb{Z}[q].$$
\end{theorem}
We emphasis that both sides of the congruence belong to $\mathbb{Z}[q]$. We briefly recall the argument 
Since $ka+b$ and $kr+s$ are integers, we may conclude that $\sqbinom{ka+b}{kr+s}_q\in \mathbb{Z}[q]$, it follows that we may rewrite this as

\begin{eqnarray}
    \sqbinom{ka+b}{kr+s}_q &=& \frac{(q;q)_{ka+b}}{(q;q)_{kr+s}(q;q)_{k(a-r)+(b-s)}} \nonumber \\
    &=& \frac{(q;q)_{ka+b}}{(1-q^k)^a}\frac{(1-q^k)^r}{(q;q)_{kr+s}}\frac{(1-q^k)^{a-r}}{(q;q)_{k(a-r)+(b-s)}} \nonumber 
\end{eqnarray}
It is enough to take a look at 

$$\frac{(q;q)_{ka+b}}{(1-q^k)^a}=\frac{(\prod_{s=0}^{a-1}(q^{ks+1};q)_{k-1})(q^k;q^k)_a(q^{ka+1};q)_b }{(1-q^k)^a}.$$
From here if we apply $q\rightarrow \zeta _k$, it follows by \cite[Proposition 2.1]{DESARMENIEN198219} that 
$$\lim_{q\rightarrow \zeta _k}\frac{(q;q)_{ka+b}}{(1-q^k)^a}=k^aa!(\zeta _k;\zeta _k)_{b}$$
which implies that
$$\frac{(q;q)_{ka+b}}{(1-q^k)^a}\equiv k^aa!(q;q)_{b}\mod \Phi _k(q)$$
over $\mathbb{Z}[q]$.
Now using this identity we can now assert the following:
\begin{eqnarray}
    \sqbinom{ka+b}{kr+s}_q
    &=& \frac{(q;q)_{ka+b}}{(1-q^k)^a}\frac{(1-q^k)^r}{(q;q)_{kr+s}}\frac{(1-q^k)^{a-r}}{(q;q)_{k(a-r)+(b-s)}} \nonumber \\
(q\rightarrow \zeta _k)    &\rightarrow & \frac{k^aa!(\zeta ; \zeta )_{b}}{k^rr!(\zeta;\zeta )_sk^{a-r}(a-r)!(\zeta;\zeta)_{b-s}} \nonumber \\
&=&\frac{a!}{r!(a-r)!}\frac{(\zeta;\zeta)_{b}}{(\zeta;\zeta)_{s}(\zeta;\zeta)_{b-s}} \nonumber \\
&=& {{a}\choose{r}} \sqbinom{b}{s}_{\zeta_k} \nonumber 
\end{eqnarray}
Note that this is the image of another element ${{a}\choose{r}} \sqbinom{b}{s}_{q}$ which also lives in $\mathbb{Z}[q]$ thus $\sqbinom{ka+b}{kr+s}_{\zeta _k}- {{a}\choose{r}}\sqbinom{b}{s}_{\zeta _k} =0$
which is equivalent to 

$$\sqbinom{ka+b}{kr+s}_q- {{a}\choose{r}}\sqbinom{b}{s}_q \in \Phi _k(q)\mathbb{Z}[q].$$
\par

\begin{remark} While this result is quite useful it's worth mentioning some of the limitations. We note that $a,b,r,s,k$ are integers which imply that the elements $ka+b$ and $kr+s$ must also be integers, that is to say we can't apply this theorem directly to $p$-adic inputs. This is the primary reason why we must work with truncated approximation of $p$-adic numbers.
\par

\end{remark}
Finally, the last piece to finalize the proof is the following identity.
\begin{lemma}\label[lemma]{binomial fact1}
    $$\sqbinom{A-i}{B}_q\sqbinom{B}{i}_q=\sqbinom{A-i}{A-B}_q\sqbinom{A-B}{i}_q$$
\end{lemma}
\begin{proof}
\begin{equation} \frac{[A-i]_q!}{[B]_q![A-B-i]_q!}\cdot \frac{[B]_q!}{[i]_q![B-i]_q!}=\frac{[A-i]_q!}{[A-B-i]_q![i]_q![B-i]_q!}=\frac{[A-i]_q!}{[A-B]_q![B-i]_q!}\cdot \frac{[A-B]_q!}{[i]_q![A-B-i]_q!} \nonumber 
\end{equation}
\end{proof}

\subsection{Main Result}
We are now ready to prove our main result. We want control over the individual truncation of $k_r/2$ and $\alpha (a,b)$ which we will denote with $k_r^{(s)}/2$ and $\alpha ^{(t)}(a,b)$. It turns out that as long as $n\leq s\leq t$ our identity will hold for any $\Phi _{p^m}$ with $1\leq m\leq n$. 
\begin{theorem}\label{main} Fix a positive integer $n$. Let $a$ and $b$ be positive even integers such that $0\leq a<b\leq p-1$. Let $m$, $s$, and $t$ be positive integers such that $1\leq m\leq n\leq s\leq t$. Finally let $i$ be an arbitrary non-negative integer, then

$$\sum_{r=0}^{2^n-1}(-1)^r\sqbinom{\frac{k^{(s)}_r}{2}}{i}_q\sqbinom{\alpha^{(t)}(a,b)-i}{\frac{k^{(s)}_r}{2}}_q\equiv 0\mod \Phi _{p^m}$$
\end{theorem}
\begin{proof}
First, consider arbitrary integers $m,s,t$ so that $1\leq m\leq n\leq s\leq t$. Consider the set $K_n$ as defined before, our goal is to show that for any $k_r\in K_n$ for a given $m$ we can find a distinct element $k_{r'}\in K_n$ so that 
$$\sqbinom{\frac{k^{(s)}_r}{2}}{i}_q\sqbinom{\alpha^{(t)}(a,b)-i}{\frac{k^{(s)}_r}{2}}_q\equiv \sqbinom{\frac{k^{(s)}_{r'}}{2}}{i}_q\sqbinom{\alpha^{(t)}(a,b)-i}{\frac{k^{(s)}_{r'}}{2}}_q\mod \Phi _{p^m}$$

We can rewrite any non-negative integer $N$ in the form
    $$N = T_0^m(N)+p^mT_m(N)$$
hence by an application of \Cref{Lucas qbinomial} we get the following identities

\begin{eqnarray}\sqbinom{\frac{k^{(s)}_r}{2}}{i}_q&=& \sqbinom{T_0^{m}(\frac{k^{(s)}_r}{2})+p^mT_m(\frac{k^{(s)}_r}{2})}{T_0^{m}(i)+p^mT_m(i)}_q \nonumber \\
&\equiv& 
{{T_m(\frac{k_r^{(s)}}{2}}\choose{T_m(i)}}\sqbinom{T_0^m(\frac{k_r^{(s)}}{2})}{T_0^m(i)}_q \mod \Phi_{p^m} \label{eq: breakuppt1} 
\end{eqnarray}

\begin{eqnarray}\sqbinom{\alpha^{(t)}(a,b)-i}{\frac{k^{(s)}_r}{2}}_q &=& \sqbinom{T_0^{m}(\alpha^{(t)}(a,b)-i)+p^mT_m(\alpha^{(t)}(a,b)-i)}{T_0^{m}(\frac{k^{(s)}_r}{2})+p^mT_m(\frac{k^{(s)}_r}{2})}_q \nonumber \\ 
&\equiv  &
{{T_m(\alpha^{(t)}(a,b)-i)}\choose{T_m(\frac{k^{(s)}_r}{2})}}\sqbinom{T_0^m(\alpha^{(t)}(a,b)-i)}{T_0^m(\frac{k^{(s)}_r}{2})}_q \mod \Phi_{p^m} \label{eq: breakuppt2}
\end{eqnarray}

we will focus our attention on the $q$-binomial portion first.
\begin{claim}
\begin{equation}\label{eq:qbinomial equality} \sqbinom{T_0^m(\frac{k^{(s)}_r}{2})}{T_0^m(i)}_q   \sqbinom{T_0^m(\alpha^{(t)}(a,b)-i)}{T_0^m(\frac{k^{(s)}_r}{2})}_q= \sqbinom{T_0^m(D_m(\frac{k^{(s)}_r}{2}))}{T_0^m(i)}_q \sqbinom{T_0^m(\alpha^{(t)}(a,b)-i)}{T_0^m(D_m(\frac{k^{(s)}_r}{2})}_q
\end{equation}
\end{claim}
\begin{proof}
In order to prove this we will separate this into two cases. We see that in general 
$$T_0^m(\alpha (a,b)-i)+T_0^m(i)=T_0^n(\alpha (a,b))+\epsilon_m$$
where $\epsilon_m$ is the carry term. The two cases are whether or not $\epsilon _m=0$. For the first case $\epsilon _m\neq 0$ this term is non-zero when $T_0^m(\alpha ^{(t)}(a,b))<T_0^m(i)$, but this would also imply that $T_0^m(k_r^{(s)}/2)<T_0^m(i)$ and $T_0^m(D_m(k_r^{(s)}/2))<T_0^m(i)$ and hence both sides are $0$ and thus we have equality. \\
For the second case $\epsilon_m =0$ so we can assume 
$$T_0^m(\alpha (a,b)-i)+T_0^m(i)=T_0^n(\alpha (a,b))$$
so we can apply \Cref{binomial fact1} and \Cref{Basics}(c) plus \Cref{BasicTruncate} to result in the following equalities

\begin{eqnarray}\sqbinom{T_0^m(\frac{k^{(s)}_r}{2})}{T_0^m(i)}_q   \sqbinom{T_0^m(\alpha^{(t)}(a,b)-i)}{T_0^m(\frac{k^{(s)}_r}{2})}_q&=&\sqbinom{T_0^m(\alpha^{(t)}(a,b))-T_0^m(\frac{k^{(s)}_r}{2})}{T_0^m(i)}_q   \sqbinom{T_0^m(\alpha^{(t)}(a,b)-i)}{T_0^m(\alpha^{(t)}(a,b))-T_0^m(\frac{k^{(s)}_r}{2})}_q \nonumber \\
&=&  \sqbinom{T_0^m(D_m(\frac{k^{(s)}_r}{2}))}{T_0^m(i)}_q \sqbinom{T_0^m(\alpha^{(t)}(a,b)-i)}{T_0^m(D_m(\frac{k^{(s)}_r}{2})}_q  . \nonumber 
\end{eqnarray}
\end{proof}

Next we consider the regular binomial coefficients and apply \Cref{Basics}(d) and \Cref{BasicTruncate} to result in 
\begin{equation}\label{eq:Binomialequal}
    {{T_m(\frac{k^{(s)}_r}{2})}\choose{T_m(i)}}{{T_m(\alpha^{(t)}(a,b)-i)}\choose{T_m(\frac{k^{(s)}_r}{2})}}={{T_m(D_m(\frac{k^{(s)}_r}{2}))}\choose{T_m(i)}}{{T_m(\alpha^{(t)}(a,b)-i)}\choose{T_m(D_m(\frac{k^{(s)}_r}{2}))}}.
\end{equation}

Applying \Cref{eq:qbinomial equality} and \Cref{eq:Binomialequal} to \Cref{eq: breakuppt1} and \Cref{eq: breakuppt2} we have established the following equality

\begin{equation}\label{result 1}\sqbinom{\frac{k^{(s)}_r}{2}}{i}_q\sqbinom{\alpha^{(t)}(a,b)-i}{\frac{k^{(s)}_r}{2}}_q\equiv \sqbinom{D_m(\frac{k^{(s)}_r}{2})}{i}_q\sqbinom{\alpha^{(t)}(a,b)-i}{D_m(\frac{k^{(s)}_r}{2})}_q\mod \Phi _{p^m}.
\end{equation}
We have $\{ k_i \vert 0\leq i \leq 2^n-1 \} = \{ k_r\in K_n \}$, since we also define $K_n$ starting from every possible combination of $a$ and $b$ in the first $n$ digits, we need a way to recover $r$, or more specifically just the first digit as that is what determines it's parity to preserve $(-1)^r$. By \Cref{Basics}(e) for any $k_r\in K_n$ we define $\gamma (k_r):=\frac{1}{b-a}(T_0^1(k_r)-a)$ and similarly $\gamma (k_r^{(s)}):=\frac{1}{b-a}(T_0^1(k_r^{(s)})-a)$, this function extracts the parity of $r$ associated to $k_r$, thus we can now state 
\begin{equation}\sum_{r=0}^{2^n-1}(-1)^r\sqbinom{\frac{k^{(s)}_r}{2}}{i}_q\sqbinom{\alpha^{(t)}(a,b)-i}{\frac{k^{(s)}_r}{2}}_q= \sum_{k_r\in K_n}(-1)^{\gamma(k^{(s)}_r)}\sqbinom{\frac{k^{(s)}_r}{2}}{i}_q\sqbinom{\alpha^{(t)}(a,b)-i}{\frac{k^{(s)}_r}{2}}_q.\nonumber \end{equation}
\\
Recall that $k^{(s)}_r$ was arbitrarily chosen from $0\leq r\leq 2^n-1$ and note that each $D_m$ is an involution. For each $k_r\in K_n$, by \Cref{Basics}(a)(e) and \Cref{BasicTruncate} it follows that $D_m(k^{(s)}_r)\in K_n$ and $k^{(s)}_r\neq D_m(k^{(s)}_r)$ (a consequence of having different parity). It follows that there exists an element $0\leq r'\leq 2^n-1$ so that $D_m(k_r^{(s)}/2)=k_{r'}^{(s)}/2$ holds. By \Cref{result 1} and \Cref{Basics}(e) along with \Cref{BasicTruncate} it follows that 

$$(-1)^{\gamma (k_r)} \sqbinom{\frac{k^{(s)}_r}{2}}{i}_q\sqbinom{\alpha^{(t)}(a,b)-i}{\frac{k^{(s)}_r}{2}}_q+(-1)^{\gamma(k_{r'})}\sqbinom{\frac{k^{(s)}_{r'}}{2}}{i}_q\sqbinom{\alpha^{(t)}(a,b)-i}{ \frac{k^{(s)}_{r'}}{2}}_q \equiv 0 \mod \Phi _{p^m}$$ 
this shows that for every element $k_r\in K_n$ there exists a distinct element $k_{r'}\in K_n$ such that their corresponding terms will sum to $0$ modulus $\Phi _{p^m}(q)$. Finally by using \Cref{Basics}(a)(b) and \Cref{BasicTruncate} we can conclude that
$$\sum_{k_r\in K_n}(-1)^{\gamma(k_r^{(s)})}\sqbinom{\frac{k^{(s)}_r}{2}}{i}_q\sqbinom{\alpha^{(t)}(a,b)-i}{\frac{k^{(s)}_r}{2}}_q \equiv 0\mod \Phi _{p^m}.$$

\end{proof}

\subsection{Some Corollaries}

Now that we have established our main result, we can see some of the immediate consequences. The first two results essentially build a link from working with $q$-binomials to working with regular binomial coefficients. 
\begin{corollary}\label[corollary]{1coro} Following the same assumptions in \Cref{main}, we have the following congruence in $\mathbb{Z}[q]$
    $$\sum_{r=0}^{2^n-1}(-1)^r \sqbinom{k_r^{(s)}/2}{i}_q\sqbinom{\alpha^{(t)}(a,b)-i}{k_r^{(s)}/2}_q\equiv 0 \mod [p^n]_q$$
\end{corollary}
\begin{proof}
From \Cref{main}, we have $$\sum_{r=0}^{2^n-1}(-1)^r \sqbinom{k_r^{(s)}/2}{i}_q\sqbinom{\alpha^{(t)}(a,b)-i}{k_r^{(s)}/2}_q \in \Phi _{p^m}(q)\mathbb{Z}[q]$$ for $1\leq m\leq n$, which leads immediately into
$$\sum_{r=0}^{2^n-1}(-1)^r \sqbinom{k_r^{(s)}/2}{i}_q\sqbinom{\alpha^{(t)}(a,b)-i}{k_r^{(s)}/2}_q\in (\prod_{m=1}^n\Phi_{p^m}(q))\mathbb{Z}[q].$$ As $\prod_{m=1}^n\Phi _{p^m}(q)=[p^n]_q$, we have yielded the desired result.
\end{proof}

\begin{corollary}\label[corollary]{2coro} Again following the same assumptions in \Cref{main} we have the following congruence over $\mathbb{Z}[q]$
    $$\sum_{r=0}^{2^n-1}(-1)^r {{k_r^{(s)}/2}\choose{i}}{{\alpha^{(t)}(a,b)-i}\choose{k_r^{(s)}/2}}\equiv 0 \mod p^n$$
\end{corollary}
\begin{proof}
We know \Cref{1coro} holds, thus 
$$\sum_{r=0}^{2^n-1}(-1)^r \sqbinom{k_r^{(s)}/2}{i}_q\sqbinom{\alpha^{(t)}(a,b)-i}{k_r^{(s)}/2}_q\in [p^n]_q\mathbb{Z}[q],$$ this implies there exists some $f(q)\in \mathbb{Z}[q]$ such that $$f(q)[p^n]_q=\sum_{r=0}^{2^n-1}(-1)^r \sqbinom{k_r^{(s)}/2}{i}_q\sqbinom{\alpha^{(t)}(a,b)-i}{k_r^{(s)}/2}_q .$$ Now $f(1)\in \mathbb{Z}$, and $[p^n]_1=p^n\in \mathbb{Z}$, thus $$\lim_{q\rightarrow 1}f(q)[p^n]_q=\lim_{q\rightarrow 1}\sum_{r=0}^{2^n-1}(-1)^r \sqbinom{k_r^{(s)}/2}{i}_q\sqbinom{\alpha^{(t)}(a,b)-i}{k_r^{(s)}/2}_q\in \in p^n\mathbb{Z}$$
and so it follows that

$$\sum_{r=0}^{2^n-1}(-1)^r {{k_r^{(s)}/2}\choose{i}}{{\alpha^{(t)}(a,b)-i}\choose{k_r^{(s)}/2}}\in p^n\mathbb{Z}.$$

\end{proof}
At this point everything we have established was for integer inputs to the binomial coefficients, our goal now is to remove the truncations and work with $p$-adic inputs.

\begin{corollary}\label[corollary]{3coro} Now we consider the results of \Cref{main}, \Cref{1coro}, and \Cref{2coro}. The following hold for an arbitrary integer $1\leq n\leq s$
    \begin{itemize}
        \item[(a)] $\sum_{r=0}^{2^n-1}(-1)^r \sqbinom{k_r^{(s)}/2}{i}_q\sqbinom{\alpha(a,b)-i}{k_r^{(s)}/2}_q\equiv 0 \mod \Phi _{p^m}(q)$ for $1\leq m\leq n$ in $\mathbb{Z}_p[q]$.
        \item[(b)] $\sum_{r=0}^{2^n-1}(-1)^r \sqbinom{k_r^{(s)}/2}{i}_q\sqbinom{\alpha(a,b)-i}{k_r^{(s)}/2}_q\equiv 0 \mod [p^n]_q$ in $\mathbb{Z}_p[q]$.
        \item[(c)] $\sum_{r=0}^{2^n-1}(-1)^r {{k_r^{(s)}/2}\choose{i}}{{\alpha(a,b)-i}\choose{k_r^{(s)}/2}}\equiv 0\mod p^n$ in $\mathbb{Z}_p$.
    \end{itemize}
\end{corollary}
\begin{remark}
    There is a subtle change here, originally for finite truncations we were fine staying completely in $\mathbb{Z}$ and $\mathbb{Z}[q]$, but now it's worth considering that if $x\in \mathbb{Z}_p$ and $k$ is a positive integer, then ${{x}\choose{k}}\in \mathbb{Z}_p$ and $\sqbinom{x}{k}\in \mathbb{Z}_p[q]$ so we will be working over these rings instead when working with $p$-adic numbers in the binomial coefficients. 
\end{remark}
\begin{proof}(a)
Fix an integer $s$ and define a sequence $t_j=s+j$ then for the sequence $\{ t_j \}_{j=0}^\infty$ we define another sequence $f_j = \sum_{r=0}^{2^n-1}(-1)^r \sqbinom{k_r^{(s)}/2}{i}_q\sqbinom{\alpha^{(t_j)}(a,b)-i}{k_r^{(s)}/2}_q$ where it is clear by \Cref{main} that $f_j=0$ over $\mathbb{Z}[q]/\Phi_{p^m}(q)$ for all $t_j$. It follows by \cite[pp.~200--201]{CONRAD2000185} that we have continuity $x\mapsto \sqbinom{x}{n}_q$ over extensions of $\mathbb{Q}_p$ hence over $\mathbb{Q}_p[\zeta _{p^m}]$ (where $\zeta _{p^m}$ is the $p^m$-th primitive root of unity) thus the following holds over $\mathbb{Z}_p[q]/\Phi_{p^m}(q)$ with $1\leq m\leq n$

$$0=\lim_{j\rightarrow \infty}f_j=\sum_{r=0}^{2^n-1}(-1)^r \sqbinom{k_r^{(s)}/2}{i}_q\sqbinom{\lim\limits_{j\rightarrow \infty} \alpha^{(t_j)}(a,b)-i}{k_r^{(s)}/2}_q=\sum_{r=0}^{2^n-1}(-1)^r \sqbinom{k_r^{(s)}/2}{i}_q\sqbinom{\alpha(a,b)-i}{k_r^{(s)}/2}_q$$
\end{proof}
\begin{proof}(b) By part $(a)$
$$\sum_{r=0}^{2^n-1}(-1)^r \sqbinom{k_r/2}{i}_q\sqbinom{\alpha(a,b)-i}{k_r/2}_q\in \Phi _{p^m}(q)\mathbb{Z}_p[q]$$ for $1\leq m\leq n$, it follows immediately that $$\sum_{r=0}^{2^n-1}(-1)^r \sqbinom{k_r/2}{i}_q\sqbinom{\alpha(a,b)-i}{k_r/2}_q\in (\prod_{m=1}^n\Phi _{p^k}(q))\mathbb{Z}_p[q]$$ and this yields the desired result as $\prod_{i=1}^n\Phi _{p^k}(q)=[p^n]_q$. 
\end{proof}

\begin{proof}(c)
As before we define the sequence $t_j=s+j$ so that we may then define the sequence $h_j = \sum_{r=0}^{2^n-1}(-1)^r {{k_r^{(s)}/2}\choose{i}}{{\alpha^{(t_j)}(a,b)-i}\choose{k_r^{(m)}/2}}$ so that $h_j=0$ by \Cref{2coro} in $\mathbb{Z}/p^n\mathbb{Z}$ hence by continuity of $x\mapsto {{x}\choose{k}}$ in $\mathbb{Z}_p$ (\cite[p.~172]{RobertPadic}) we can conclude that the following holds over $\mathbb{Z}_p/p^n\mathbb{Z}_p$
$$0 = \lim\limits_{j\rightarrow \infty}h_j=\sum_{r=0}^{2^n-1}(-1)^r{{k_r^{(s)}/2}\choose{i}} {{\lim\limits_{j\rightarrow\infty } \alpha ^{t_j}(a,b)-i}\choose{k_r^{(s)}/2}} = \sum_{r=0}^{2^n-1}(-1)^r{{k_r^{(s)}/2}\choose{i}} {{\alpha (a,b)-i}\choose{k_r^{(s)}/2}}.$$
\end{proof}
\begin{corollary}\label[corollary]{when i=0} We keep the same assumptions from before. The following hold
    \begin{itemize}
        \item[(a)] $\sum_{r=0}^{2^n-1}(-1)^r \sqbinom{\alpha(a,b)}{k_r^{(s)}/2}_q\equiv 0 \mod \Phi _{p^m}(q)$ for $1\leq m\leq n$ in $\mathbb{Z}_p[q]$
        \item[(b)] $\sum_{r=0}^{2^n-1}(-1)^r \sqbinom{\alpha(a,b)}{k_r^{(s)}/2}_q\equiv 0 \mod [p^n]_q$ in $\mathbb{Z}_p[q]$
        \item[(c)] $\sum_{r=0}^{2^n-1}(-1)^r {{\alpha(a,b)}\choose{k_r^{(s)}/2}}\equiv 0\mod p^n$ in $\mathbb{Z}_p$
    \end{itemize}
\end{corollary}
\begin{proof}
Part (a) follows from \Cref{3coro}(a) and $i=0$, part (b) follows from \Cref{3coro}(b) and $i=0$, and part (c) follows from \Cref{3coro}(c) and $i=0$. 
\end{proof}

\section{Concluding remarks}
Now we can show how this result generalizes \Cref{ak}. Consider \Cref{3coro}(c) if we set $a=0$ and $b=p-1$, by the definition of $k_r$ and $\alpha (a,b)$ we have $k_r=(p-1)\sum_{i=0}^\infty r_ip^i$ and $\alpha (a,b)=-\frac{1}{2}$ which corresponds with the same definition used in \Cref{ak}, then we have the following results (note that since $a=0$ we have $k_r=k_r^{(s)}$ for all $n\leq s$)

$$\sum_{r=0}^{2^n-1}(-1)^r{{k_r/2}\choose{i}} {{-1/2-i}\choose{k_r/2}}\equiv 0\mod p^n.$$ \par

 As the original supercongruence of \cite{ichino2025representationsmathrmgl2mathbbzpnmathbbzsupercongruences} is related to the representation theory of $\mathrm{GL}_2(\mathbb{Z}/p^n\mathbb{Z})$,  perhaps there is some $q$-analog of this representation theoretic connection. We hope to consider this in future work. 

\printbibliography

\end{document}